\documentclass[11pt]{amsart}
\usepackage{amssymb,latexsym}
\newtheorem{theorem}{Theorem}[section]
\newtheorem{lemma}[theorem]{Lemma}
\newtheorem{proposition}[theorem]{Proposition}
\newtheorem{corollary}[theorem]{Corollary}
\theoremstyle{definition}
\newtheorem{definition}[theorem]{Definition}

\theoremstyle{remark}

\numberwithin{equation}{section}

\newcommand{\blankbox}[2]

\begin{document}
\setlength{\baselineskip}{1.2\baselineskip}
\title  [Neumann problem  for  elliptic equations]
{ The Monge-Amp\`ere equation for strictly
 $(n-1)$-convex functions with Neumann condition}
\author{Bin Deng}
\address{Department of Mathematics\\
         University of Science and Technology of China\\
         Hefei, 230026, Anhui Province, China.}
\email{bingomat@mail.ustc.edu.cn}

\thanks{$*$ Research supported by NSFC No.11721101 and No.11871255.
I would like to thank professor Xi-Nan Ma,
my advisor, for his constant encouragement and guidance.}
\begin{abstract}
  A $C^2$ function on $\mathbb{R}^n$ is called strictly $(n-1)$-convex if the sum of any
  $n-1$ eigenvalues of its Hessian is positive. In this paper, we establish a global $C^2$ estimates to the Monge-Amp\`ere equation for strictly
 $(n-1)$-convex functions with Neumann condition. By the method of continuity, we prove an existence theorem for strictly $(n-1)$-convex
  solutions of the Neumann problems.
\end{abstract}
\keywords{Neumann problem, $(n-1)$-convex, elliptic equation}

\maketitle

%

%

\section{Introduction}
Let $\Omega\subset\mathbb{R}^n$ be a bounded convex domain and $\nu(x)$ be the outer unit
normal at $x\in\partial\Omega$.
Suppose $f\in C^2(\Omega)$ is positive and $\phi\in C^3(\overline{\Omega})$.
In this paper, we mainly consider  the following equations of Monge-Amp\'ere type with
Neumann condition,
\begin{equation}\label{eq}
\left\{
\begin{aligned}
  &\det(W)=f(x),\quad\text{in}\quad \Omega,\\
  &\frac{\partial u}{\partial \nu}=-u+\phi(x),\quad\text{in}\quad \partial\Omega.
\end{aligned}
\right.
\end{equation}
where the matrix $W=(w_{\alpha_{1}\cdot\cdot\cdot\alpha_{m},\beta_{1}\cdot\cdot\cdot\beta_{m}})_{n\times n}$, for $m= n-1$, with the elements as follows,
\begin{eqnarray}\label{w0}
w_{\alpha_{1}\cdot\cdot\cdot\alpha_{m},\beta_{1}\cdot\cdot\cdot\beta_{m}}
=\sum^{m}_{i=1}\sum^{n}_{j=1}u_{\alpha_{i} j
}\delta^{\alpha_{1}\cdot\cdot\cdot\alpha_{i-1}j\alpha_{i+1}
\cdot\cdot\cdot\alpha_{m}}_{\beta_{1}\cdot\cdot\cdot
\beta_{i-1}\beta_{i}\beta_{i+1}\cdot\cdot\cdot\beta_{m}},
\end{eqnarray}
a linear combination of $u_{ij}$, where $u_{ij}=\frac{\partial^2u}{\partial x_i\partial x_j}$ and $\delta^{\alpha_{1}\cdot\cdot\cdot\alpha_{i-1}\gamma\alpha_{i+1}\cdot\cdot\cdot\alpha_{m}}_{\beta_{1}\cdot\cdot\cdot
\beta_{i-1}\beta_{i}\beta_{i+1}\cdot\cdot\cdot\beta_{m}}$ is the generalized Kronecker symbol. All indexes $i,j,\alpha_i,\beta_i,\cdots$ come from $1$ to $n$.

For general $1\leq m\leq n-1$, the matrix $W\in\mathbb{R}^{C_n^m\times C_n^m}$,
$C_n^m=\frac{n!}{m!(n-m)!}$,
 comes from the following operator $U^{[m]}$ as in \cite{cns2} and \cite{hmw1}.
 First, note
that $(u_{ij})_{n\times n}$ induces an operator $U$ on $\mathbb{R}^n$ by
\begin{eqnarray}
  U(e_{i})=\sum_{j=1}^{n}u_{i j}e_{j},\quad \forall 1\leq i\leq n,\nonumber
\end{eqnarray}
where $\{e_1,e_2,\cdots,e_n\}$ is the standard basis of $\mathbb{R}^n$.  We further extend $U$ to acting on the real
 vector space $\wedge^m\mathbb{R}^n$ by
\begin{eqnarray}
  U^{[m]}(e_{\alpha_1}\wedge\cdots\wedge e_{\alpha_m})=
  \sum_{i=1}^{m}e_{\alpha_1}\wedge\cdots\wedge U(e_{\alpha_i})\wedge\cdots\wedge e_{\alpha_m},\nonumber
\end{eqnarray}
where $\{e_{\alpha_1}\wedge\cdots\wedge e_{\alpha_m}\ | \ 1\leq \alpha_1<\cdots<\alpha_m\leq n\}$ is the standard basis
for $\wedge^m\mathbb{R}^n$. Then $W$ is the matrix of $U^{[m]}$ under this standard basis.
It is convenient to denote the multi-index by $\overline{\alpha}=(\alpha_1\cdots\alpha_m)$. We only consider the admissible multi-index, that is, $1\leq\alpha_1<\alpha_2,\cdots<\alpha_m\leq n$. By the dictionary arrangement, we can arrange all admissible multi-indexes from $1$ to $C_n^m$, and use $N_{\overline{\alpha}}$ denote the order number of the multi-index $\overline{\alpha}=(\alpha_1\cdots\alpha_n)$, i.e., $N_{\overline{\alpha}}=1$ for $\overline{\alpha}=(12\cdots m)$, $\cdots$. We also use $\overline{\alpha}$ denote the index set $\{\alpha_1,\cdots,\alpha_n\}$.
It is not hard to see that
\begin{eqnarray}
W_{N_{\overline{\alpha}}N_{\overline{\alpha}}}= w_{\overline{\alpha},\overline{\alpha}}=\sum_{i=1}^{m}u_{\alpha_i\alpha_i}\label{w1}
\end{eqnarray}
and
\begin{eqnarray}
W_{N_{\overline{\alpha}}N_{\overline{\beta}}}= w_{\overline{\alpha}\overline{\beta}}=(-1)^{|i-j|}u_{\alpha_i\beta_j},\label{w2}
\end{eqnarray}
if the index set $\{\alpha_{1},\cdot\cdot\cdot,\alpha_{m}\}\setminus\{\alpha_i\}$ equals to the index set $\{\beta_{1},\cdot\cdot\cdot,\beta_{m}\}\setminus\{\beta_j\}$ but $\alpha_i\neq \beta_j$
; and also
\begin{eqnarray}
 W_{N_{\overline{\alpha}}N_{\overline{\beta}}}= w_{\overline{\alpha}\overline{\beta}}=0,\label{w3}
\end{eqnarray}
if the index sets $\{\alpha_{1},\cdot\cdot\cdot,\alpha_{m}\}$ and $\{\beta_{1},\cdot\cdot\cdot,\beta_{m}\}$
are differed by more than one elements. Specifically, for $n=3, m=2$, we have
\begin{eqnarray*}
W=
\left(
  \begin{array}{ccc}
    u_{11}+u_{22} & u_{23} & -u_{13} \\
    u_{32} & u_{11}+u_{33} & u_{12} \\
    -u_{31} & u_{21}& u_{22}+u_{33}\\
  \end{array}
\right)
\end{eqnarray*}

It follows that $W$ is symmetric and is diagonal if $(u_{ij})_{n\times n}$ is diagonal.
The eigenvalues of  $W$ are the sums of eigenvalues of $(u_{ij})_{n\times n}$.
Denoted by $\mu(D^2u)=(\mu_1,\cdots,\mu_n)$ the eigenvalues of the
Hessian  and by $\lambda(W)=(\lambda_{1}, \lambda_{2}, \cdots, \lambda_{C_{n}^{m}})$  the eigenvalues of $W$.
Generally, for any $k=1, 2, \cdots, C_{n}^{m}$, we define the $k^{th}$ elementary symmetry
function by
\begin{eqnarray}
  S_{k}(W)=S_{k}\big(\lambda(W)\big)=\sum_{1\leq i_{1}<i_{2}<\cdots<i_{k}\leq C_{n}^{m}}\lambda_{i_{1}}\lambda_{i_{2}}\cdots\lambda_{i_{k}},\nonumber
\end{eqnarray}
 We also set $S_{0}=1$. In particular, we have
 \begin{eqnarray}
   \det(W)=S_n(W)&=&\lambda_1\lambda_2\cdots\lambda_{C_n^m}\nonumber\\
   &=&\prod_{1\leq i_{1}<i_{2}<\cdots<i_{m}\leq n}(\mu_{i_1}+\mu_{i_2}+\cdots+\mu_{i_m}).\nonumber
 \end{eqnarray}
 If $m=1$, the equation (\ref{eq}) is known as Monge-Amp\'ere equation.

Define the Garding's cone in $\mathbb{R}^n$ as
\begin{eqnarray}
\Gamma_{k} = \{\mu\in\mathbb{R}^{n}|\  S_{i}(\mu) > 0, \forall 1 \leq i \leq k\}.\nonumber
\end{eqnarray}
Then we define the generalized Garding's cone as, $1\leq m\leq n$, $1\leq k\leq C_n^m$,
\begin{eqnarray}
  \Gamma_k^{(m)}=\{\mu\in\mathbb{R}^n|\ \{\mu_{i_1}+\cdots+\mu_{i_m}|\ 1\leq i_1<\cdots<i_m\leq n\}\in\Gamma_k \ \text{in}\  \mathbb{R}^{C_n^m}\}.\nonumber
\end{eqnarray}
Obviously, $\Gamma_k=\Gamma_k^{(1)}$ and $\Gamma_n\subset\Gamma_k^{(m)}\subset\Gamma_1$.
 Normally, we say a $C^2$ function $u$ is convex if any eigenvalue of the Hessian
 is nonnegative, equivalently $\mu(D^2u)\in\overline{\Gamma_n}$.
Similarly, we give the following definition of $m$-convexity.
 \begin{definition}\label{def1}
   We say a $C^2$ function $u$ is strictly $m$-convex if $\mu(D^2u)\in \Gamma_{C_n^m}^{(m)}$, i.e.,
   the sum of any $m$ eigenvalues of the Hessian is positive.
   Furthermore, we say $u$ is $m$-convex if $\mu(D^2u)\in
   \overline{\Gamma_{C_n^m}^{(m)}}$, i.e., the sum of any $m$ eigenvalues of the Hessian
   is nonnegative.
 \end{definition}
 In particular, if  $\mu(D^2u)\in\Gamma_{n}^{(n-1)}$ for any $x\in\Omega$,
 then equivalently $\lambda(W)\in\Gamma_{n}$, such that the equation (\ref{eq}) is elliptic (see \cite{cns2} or \cite{l2}). In addition,
 we say $u$ is a strictly $(n-1)$-convex
 solution if $u$ is a solution of (\ref{eq}).


For the Dirichlet problem in $\mathbb{R}^{n}$, many results is known. For example, the Dirichlet
problem of Laplace equation is studied in \cite{gt}, Caffarelli-Nirenberg-Spruck \cite{cns} and Ivochkina
\cite{ivo} solved the Dirichlet problem of Monge-Amp\`ere equation, and Caffarelli-Nirenberg-Spruck \cite{cns2} solved the Dirichlet problem of general Hessian equations even including the case considered here. For the general Hessian
quotient equation, the Dirichlet problem is solved by Trudinger in \cite{tru}. Finally, Guan \cite{guanbo} treated the Dirichlet
problem for general fully nonlinear elliptic equation on the Riemannian manifolds without any geometric restrictions to the boundary.

Also, the Neumann or oblique derivative problem of partial differential equations was
widely studied. For a priori estimates and the existence theorem of Laplace equation
with Neumann boundary condition, we refer to the book \cite{gt}. Also, we can see the
book written by Lieberman \cite{l} for the Neumann or oblique derivative problem of linear
and quasilinear elliptic equations. In 1987, Lions-Trudinger-Urbas solved the Neumann
problem of Monge-Amp\`ere equation in the celebrated paper \cite{ltu}. For the the Neumann
problem of k-Hessian equations, Trudinger \cite{tru2} established the existence theorem when
the domain is a ball, and he conjectured (in \cite{tru2}, page 305) that one can solve the problem
in sufficiently smooth uniformly convex domains. Recently, Ma and Qiu \cite{mq} gave a positive
answer to this problem and solved the the Neumann problem of k-Hessian equations in
uniformly convex domains. After their work, the research on the Neumann problem of other equatios
has made progresses(see \cite{mx} \cite{cz} \cite{cmw} \cite{w}).

Similarly to $m$-convexity for the Hessian (see Definition \ref{def1}),
we can formulate the notion of $m$-convexity for curvature operator and second fundamental
forms of hypersurfaces. There are large amount literature in differential geometry on this
subject. For example, Sha \cite{sh} and Wu \cite{wu} introduced the $m$-convexity of the sectional
curvature of Riemannian manifolds and studied the topology for these manifolds. In a series interesting papers, Harvey and Lawson
\cite{hl1} \cite{hl2} \cite{hl3}  introduce some generally convexity on the solutions of the nonlinear elliptic Dirichlet
problem, $m$-convexity is a special case. Han-Ma-Wu \cite{hmw1} obtained  an existence theorem of $m$-convex starshaped hypersurface
with prescribed mean curvature.
 More recently, in the complex space $\mathbb{C}^n$ case, Tosatti and Weinkove\cite{tw} \cite{tw2} solved the Monge-Amp\`ere equation for
 $(n-1)$-plurisubharmonic functions on a compact K\"ahler manifold, where the $(n-1)$-plurisubharmonicity means
the sum of any $n-1$ eigenvalues of the complex Hessian  is nonnegative.

From the above geometry and analysis reasons, it is naturally to study the Neumann problem (\ref{eq}).

This paper is a sequel to \cite{Deng}. In \cite{Deng}, the author considered the following Neumann problem for general fully nonlinear equations
\begin{equation}\label{eq1}
  \left\{
  \begin{aligned}
  &S_{k}(W)=f(x),\quad\text{in}\ \Omega,\\
  &\frac{\partial u}{\partial\nu}=-u+\phi(x),\quad\text{on}\ \partial\Omega.
  \end{aligned}
  \right.
\end{equation}
The equation (\ref{eq}) is a special case of (\ref{eq1}) when $m=n-1, k=n$. Parallel to Definition \ref{def1}, we give
\begin{definition}
  We say $u$ is  $k$-admissible  if $\mu(D^2u)\in \Gamma_k^{(m)}$.
  Particularly, if $k=C_n^m$, $u$ is strictly $m$-convex.
\end{definition}
For $k\leq C_{n-1}^{m-1}=\frac{m}{n}C_n^m$, we obtained an existence theorem of the $k$-admissible solution with less geometric
restrictions to the boundary. For $m<\frac{n}{2}$ and $ k=C_{n-1}^{m-1}+k_0\leq \frac{n-m}{n}C_n^m$, we got an existence theorem
if $\Omega$ is strictly $(m,k_0)$-convex, i.e., $\kappa\in\Gamma_{k_0}^{(m)}$, where
$\kappa = (\kappa_1 ,\cdots ,\kappa_{n-1} )$ denote the principal curvatures
of $\partial\Omega$ with respect to its inner normal $-\nu$. We didn't prove the existence
for strictly $m$-convex solution for the equation (\ref{eq1}) in \cite{Deng}.
Particularly, for $m=n-1$ (maybe the most interesting case except the case $m=1$),
we got the existence of the $k$-admissible solution for $k\leq n-1$
only except that of the $(n-1)$-convex solution for $k=n$. In this paper, given a strong
geometric restriction to the boundary,
we can prove the existence of strictly $(n-1)$-convex
solution to the Neumann problem (\ref{eq}).


In this paper, we always denote $\kappa=(\kappa_1,\cdots,\kappa_{n-1})$
the principal curvature and $H=\sum\limits_{i=1}^{n-1}\kappa_i$ the
mean curvature of the boundary.
We now state the main results of this paper as follows.
\begin{theorem}\label{th1.1}
   Suppose $\Omega\subset\mathbb{R}^{n}\ (n\geq3)$ is a bounded
   strictly convex domain  with $C^{4}$ boundary. Denote $\kappa_{max}(x)
 $ ($\kappa_{min}(x)$) the maximum (minimum) principal curvature at $x\in\partial\Omega$ such that $\kappa_{max}-\kappa_{min}<\frac{H}{2(n-1)(n-2)}$. Let $f \in C^{2}(\Omega)$ is a positive function and $\phi\in C^{3}(\overline{\Omega})$. Then there exists a unique strictly $(n-1)$-convex solution $u \in C^{3,\alpha}(\overline{\Omega})$ of the Neumann problem (\ref{eq}).
\end{theorem}

When the dimension $n$ is large, we require the domain $\Omega$ is almost a ball.
As a special case, for $n=3$, $H=\kappa_{max}+\kappa_{min}$, we have
\begin{corollary}\label{co1.2}
   Suppose $\Omega\subset\mathbb{R}^{3}$ is a bounded
   strictly convex domain  with $C^{4}$ boundary. Denote $\kappa_{max}(x)
 $ ($\kappa_{min}(x)$) the maximum (minimum) principal curvature at $x\in\partial\Omega$ such that $\kappa_{max}<\frac{5}{3}\kappa_{min}$. Let $f \in C^{2}(\Omega)$ is a positive function and $\phi\in C^{3}(\overline{\Omega})$. Then there exists a unique strictly $2$-convex solution $u \in C^{3,\alpha}(\overline{\Omega})$ of the Neumann problem (\ref{eq}).
\end{corollary}

The rest of this paper is arranged as follows. In section \ref{sec2}, we give some basic properties
of the elementary symmetric functions and some notations.
In section \ref{sec3},
we establish a priori $C^0$ estimates and global gradient estimates.
In section \ref{sec4}, we show the proof of the global estimates of second order derivatives.
Finally, we can prove the existence theorem by the method of continuity in section \ref{sec5}.

\section{Preliminary}\label{sec2}
In this section, we give some basic properties of elementary symmetric functions and some
notations.

First, we denote by
$S_{k}(\lambda|i)$ the symmetric function with $\lambda_{i} = 0$ and $S_{k}(\lambda|ij)$ the symmetric function with
$\lambda_{i} = \lambda_{j} = 0$.
\begin{proposition}\label{pro1}
  Let $\lambda=(\lambda_{1}, \cdots, \lambda_{n})\in \mathbb{R}^{n}$ and $k=1, \cdots, n$, then
  \begin{eqnarray}
    &&\sigma_{k}(\lambda)=\sigma_{k}(\lambda|i)+\lambda_{i}\sigma_{k-1}(\lambda|i),\quad \forall 1\leq i\leq n,\\
    &&\sum_{i=1}^{n}\lambda_{i}\sigma_{k-1}(\lambda|i)=k\sigma_{k}(\lambda),\\
    &&\sum_{i=1}^{n}\sigma_{k}(\lambda|i)=(n-k)\sigma_{k}(\lambda).
  \end{eqnarray}
\end{proposition}

We denote by $S_{k}(W|i)$ the symmetric function with $W$ deleting the $i$-row and
$i$-column and $S_{k}(W|ij)$ the symmetric function with $W$ deleting the $i, j$-rows and $i, j$-columns.
Then we have the following identities.

\begin{proposition}\label{pro2}
  Suppose $A=(a_{ij})_{n\times n}$ is diagonal, and $k$ is a positive integer, then
  \begin{equation}\label{s1}
    \frac{\partial S_{k}(A)}{\partial a_{ij}}=
    \left\{
    \begin{aligned}
    &S_{k-1}(A|i),&&\quad\text{if} \ i=j,&\\
    &0,&&\quad\text{if} \ i\neq j.&
    \end{aligned}
    \right.
  \end{equation}
  Furthermore, suppose $W=(w_{\overline{\alpha}\overline{\beta}})_{C_n^m\times C_n^m}$ defined as in (\ref{w0}) is diagonal, then
  \begin{equation}\label{sw1}
    \frac{\partial S_{k}(W)}{\partial u_{ij}}=
    \left\{
    \begin{aligned}
    &\sum_{i\in\overline{\alpha}}S_{k-1}(W|N_{\overline{\alpha})},&&\quad\text{if} \ i=j,&\\
    &0,&&\quad\text{if} \ i\neq j.&
    \end{aligned}
    \right.
  \end{equation}
\end{proposition}
\begin{proof}
  For (\ref{s1}), see a proof in  \cite{l2}.

  Note that
  \begin{eqnarray}
    \frac{\partial S_k(W)}{\partial u_{ij}}=\sum_{\overline{\alpha},\overline{\beta}}\frac{\partial S_k(W)}{\partial w_{\overline{\alpha}\overline{\beta}}}\frac{\partial w_{\overline{\alpha}\overline{\beta}}}{\partial u_{ij}},
  \end{eqnarray}
   Using (\ref{w1}), (\ref{w2}), and (\ref{w3}), (\ref{sw1}) is immediately a consequence of (\ref{s1}).
\end{proof}

  Recall that the Garding's cone is defined as
  \begin{eqnarray*}
    \Gamma_{k}=\{\lambda\in\mathbb{R}^{n} |\  S_{i}(\lambda)>0, \forall \ 1\leq i\leq k\}.
  \end{eqnarray*}
\begin{proposition}\label{pro3}
  Let $\lambda \in\Gamma_{k}$ and $k\in\{1, 2, \cdots, n\}$. Suppose that
  \begin{eqnarray*}
    \lambda_{1}\geq\cdots\geq\lambda_{k}\geq\cdots\geq\lambda_{n},
  \end{eqnarray*}
  then we have
  \begin{eqnarray}
     &&S_{k-1}(\lambda|n)\geq \cdots \geq S_{k-1}(\lambda|k) \geq \cdots \geq S_{k-1}(\lambda |1) >0,\label{2.1}\\
     &&\lambda_1 S_{k-1} (\lambda |1) \geq \frac{k}{n} S_k(\lambda),\label{2.3}\\
     &&S_{k}^{\frac{1}{k}}(\lambda) \ \text{is concave in} \ \Gamma_k.\label{2.8}
  \end{eqnarray}
  where $C^k_n=\frac{n!}{k!(n-k)!}$.
\end{proposition}
\begin{proof}
  All the properties are well known. For example, see \cite{l2} or \cite{hs} for a proof of (\ref{2.1}),
  \cite{cw} or \cite{hmw} for (\ref{2.3}) and \cite{cns2} for (\ref{2.8}).
\end{proof}

The Newton-Maclaurin inequality is as follows,

\begin{proposition}\label{pro4}
  For  $\lambda\in\Gamma_{k}$ and $k>l\geq0$, we have
  \begin{eqnarray}\label{2.7}
    \big(\frac{S_{k}(\lambda)}{C_n^k}\big)^{\frac{1}{k}}\leq\big(\frac{S_{l}(\lambda)}{C_n^l}\big)^{\frac{1}{l}},
  \end{eqnarray}
  where $C_n^k=\frac{n!}{k!(n-k)!}$. Furthermore we have
  \begin{eqnarray}\label{2.10}
  \sum_{i=1}^{n}\frac{\partial S_{k}^{\frac{1}{k}}}{\partial\lambda_{i}}\geq[C_n^k]^{\frac{1}{k}}.
  \end{eqnarray}
\end{proposition}
\begin{proof}
  See \cite{s} for a proof of (\ref{2.7}). For (\ref{2.10}), we use (\ref{2.7}) and Proposition \ref{pro1} to get
  \begin{eqnarray}
    \sum_{i=1}^n\frac{\partial S_k^{\frac{1}{k}}(\lambda)}{\partial\lambda_i}=\frac{1}{k}S_k^{\frac{1}{k}-1}\sum^n_{i=1}S_{k-1}(\lambda|i)=
    \frac{n-k+1}{k}S_{k}^{\frac{1}{k}-1}S_{k-1}(\lambda)\geq[C_n^k]^{\frac{1}{k}}.\nonumber
  \end{eqnarray}
\end{proof}

We define
\begin{eqnarray}
  d(x)=dist(x,\partial\Omega),\nonumber\\
  \Omega_{\mu}=\{x\in\Omega|\ d(x)<\mu\}.
\end{eqnarray}
 It is well known that there exists a small positive universal
constant $\mu_0$ such that $d(x)\in C^k(\Omega_{\mu}),\ \forall0<\mu\leq\mu_0$, provided $\partial\Omega\in C^k$.  As in Simon-Spruck \cite{ss} or Lieberman \cite{l}
(in page 331), we can extend $\nu$ by $\nu=-D d$ in $\Omega_{\mu}$ and note that $\nu$ is a $C^2(\overline{\Omega_{\mu}})$ vector field. As mentioned in the book \cite{l}, we also have the following formulas
\begin{eqnarray}
  |D\nu|+|D^2\nu|\leq C(n,\Omega),\quad\text{in}\ \Omega_{\mu},\nonumber\\
  \sum\limits_{i=1}^n\nu^iD_j\nu^i=\sum\limits_{i=1}^n\nu^iD_i\nu^j=\sum\limits_{i=1}^nd_id_{ij}=0,\ |\nu|=|Dd|=1,\quad\text{in}\ \Omega_{\mu}.\label{d}
\end{eqnarray}

\section{The zero-order and first-order estimates }\label{sec3}
As proved in \cite{Deng}, we have the following theorem.
\begin{theorem}\label{th3.1}
  Let $\Omega\subset \mathbb{R}^{n}\ (n\geq3)$ be a bounded domain with $C^{3}$ boundary.
  Let $f \in C^{1}(\Omega)$ is a positive function and $\phi\in C^{3}(\overline{\Omega})$. Suppose that $u\in C^{2}(\overline{\Omega})\cap C^{3}(\Omega)$
  is an $k$-admissible solution of the Neumann problem (\ref{eq1}).
   Then there exists a constant $C_1$
   depends only on $k$, $n$, $|f|_{C^1}$, $|\phi|_{C^3}|$ and $\Omega$, such that
\begin{eqnarray}\label{3.1}
  \sup_{\overline{\Omega}}(|u|+|Du|)\leq C_1.
\end{eqnarray}
\end{theorem}
\begin{proof}
See Theorem 3.1 in \cite{Deng} for the zero-order estimate. See Theorem 4.2 and Theorem
4.4 in \cite{Deng} for the first-order estimate.
\end{proof}

\section{Global Second Order Derivatives Estimates}\label{sec4}
\medskip

Generally, the double normal estimates are the most important and hardest  parts for the Neumann problem.
As in \cite{ltu} and \cite{mq},
we construct sub and super barrier function to give lower and upper bounds for $u_{\nu\nu}$ on the boundary. Then we give the global
second order estimates.

In this section, we establish the following global second order estimate.
\begin{theorem}\label{th5.1}
Suppose $\Omega\subset\mathbb{R}^{n}\ (n\geq3)$ is a bounded strictly
  convex domain  with $C^{4}$ boundary, $ m= n-1$. Denote $\kappa_{max}(x)
 $ ($\kappa_{min}(x)$) the maximum (minimum) principal curvature at $x\in\partial\Omega$ such that $\kappa_{max}-\kappa_{min}<\frac{\gamma H}{2(n-1)(n-2)}$
 for any
 $\gamma\in[\frac{1}{2},1)$. Let $f(x,z)\in C^{2}(\Omega\times \mathbb{R})$ is a positive function and $\phi(x,z)\in C^{3}(\overline{\Omega}\times \mathbb{R})$ is decreasing with $z$. If $u \in C^{3,\alpha}(\overline{\Omega})$ is a strictly $(n-1)$-convex solution of the Neumann problem
 \begin{eqnarray}\label{eq2}
   \left\{
   \begin{aligned}
     \det(W)=f(x,u),\quad\text{in}\ \Omega,\\
     \frac{\partial u}{\partial \nu}=\phi(x,u),\quad\text{on}
     \ \partial\Omega.
   \end{aligned}
   \right.
 \end{eqnarray}
 Then we have
\begin{eqnarray}
\sup_{\overline{\Omega}}|D^{2}u|\leq C,\label{5.5}
\end{eqnarray}
where $C$ depends only on $n$, $m$, $k$, $\gamma$, $|u|_{C^{1}(\overline{\Omega})}$,$|f|_{C^{2}(\overline{\Omega}\times[-M_{0},M_{0}])}$, $\min f$, $|\phi|_{C^{3}(\overline\Omega\times[-M_{0},M_{0}])}$ and $\Omega$, where $M_{0}=\sup\limits_{\Omega}|u|$.
\end{theorem}

Throughout the rest of this paper, we always admit the Einstein's summation convention. All repeated indices come from 1 to n.
We will denote $F(D^2u)=\det(W)$ and
\begin{eqnarray*}
  F^{ij}=\frac{\partial F(D^2u)}{\partial u_{ij}}=\frac{\partial \det(W)}{\partial w_{\overline{\alpha}\overline{\beta}}}
\frac{\partial w_{\overline{\alpha}\overline{\beta}}}{\partial u_{ij}}.
\end{eqnarray*}
From (\ref{w1}) and (\ref{sw1}) in Proposition \ref{pro2} we have, for any $1\leq j\leq n$,
\begin{eqnarray}
F^{ii}=\sum_{i\in\overline{\alpha}}\frac{\partial \det(W)}{\partial w_{\overline{\alpha}\overline{\alpha}}}
=\sum_{i\in\overline{\alpha}}S_{n-1}(W|N_{\overline{\alpha}}).\label{F1}
\end{eqnarray}
Throughout the rest of the paper, we will denote $\mathcal{F}=\sum\limits_{i=1}^{n}F^{ii}
=(n-1)\sum\limits_{N_{\overline{\alpha}}=1}^{n}S_{n-1}(W|N_{\overline{\alpha}})$ for simplicity.

 \subsection{Reduce the global second derivative estimates into double normal derivatives estimates on boundary}

Using the method of Lions-Trudinger-Urbas \cite{ltu}, we can reduce the second derivative estimates of the solution into the  boundary double normal estimates.

\begin{lemma}\label{le4.1}
 Let $\Omega\subset \mathbb{R}^n$ be a bounded strictly
 convex domain with $C^4$ boundary. Assume $f(x,z)\in C^{2}(\overline{\Omega}\times \mathbb{R})$ is positive and $\phi(x,z)\in C^{3}(\overline{\Omega}\times\mathbb{R})$ is
 decreasing with $z$. If $u$ is a strictly $(n-1)$-convex solution of the Neumann problem (\ref{eq2}), denote $N=\sup\limits_{\partial\Omega}|u_{\nu\nu}|$, then we have
 \begin{eqnarray}\label{4.1}
 \sup_{\overline{\Omega}}|D^2u|\leq C_{0}(1+N).
 \end{eqnarray}
where $C_0$ depends on $n$, $m$, k, $|u|_{C^{1}(\overline{\Omega})}$,
$|f|_{C^{2}(\overline{\Omega}\times [-M_{0},M_{0}])}$,
$\min f$, $|\phi|_{C^{3}(\overline{\Omega}\times[-M_{0},M_{0}])}$ and $\Omega$.
\end{lemma}

\begin{proof}
Write equation (\ref{eq2}) in the form of
\begin{equation}\label{eq3}
\left\{
\begin{aligned}
 \det(W)^{\frac{1}{n}}=\widetilde{f}(x,u),&&\text{in}\ \ \Omega,\\
\frac{\partial u}{\partial\nu}=\phi(x,u),&&\ \text{on}\ \ \partial\Omega.
\end{aligned}
\right.
\end{equation}
where $\widetilde{f}=f^{\frac{1}{n}}$.
Since $\lambda(W)\in\Gamma_n\subset\Gamma_2$ in $\mathbb{R}^{n}$, we have
 \begin{eqnarray}
   \sum_{i\neq j}|u_{ij}|\leq c(n)S_1(W)=mc(n)S_1(D^2u),\label{4.2}
 \end{eqnarray}
 where $c(n)$ is a universal number independent of $u$.
 It is sufficiently to prove (\ref{4.1}) for any direction $\xi\in\mathbb{S}^{n-1}$, that is
  \begin{eqnarray}
 u_{\xi\xi}\leq C_{0}(1+N).\label{4.3}
 \end{eqnarray}

 We consider the following auxiliary function in $\Omega\times \mathbb{S}^{n-1}$,
 \begin{eqnarray}
   v(x,\xi)=u_{\xi\xi}-v'(x,\xi)+K_1|x|^2+K_2|Du|^2,\label{4.4}
 \end{eqnarray}
 where $v'(x,\xi)=a^{l}u_{l}+b= 2(\xi\cdot\nu)\xi'\cdot(\phi_{x_{l}}+\phi_{z}u_{l}-u_{l}D\nu^{l})$, with  $\xi'=\xi-(\xi\cdot\nu)\nu$ and $a^{l}=2(\xi\cdot\nu)(\xi'^{l}\phi_{z}-\xi'^{i}D_{i}\nu^{l})$. $K_{1}$, $K_{2}$ are positive constants to be determined. By a direct computation, we have
 By direct computations, we have
  \begin{eqnarray}
    v_{i}&=&u_{\xi\xi i}-D_{i}a^{l}u_{l}-a^{l}u_{ii}-D_{i}b+2K_{1}x_{i}+2K_{2}u_{l}u_{li},\label{4.5}\\
    v_{ij}&=& u_{\xi\xi ij}-D_{ij}a^{l}u_{l}-D_{i}a^{l}u_{lj}-D_{j}a^{l}u_{li}-a^{l}u_{lij}-D_{ij}b\nonumber\\
    &&+2K_{1}\delta_{ij}+2K_{2}u_{li}u_{lj}+2K_{2}u_{l}u_{lij}.\label{4.6}
    \end{eqnarray}

  Denote $\widetilde{F}(D^2u)=\det(W)^{\frac{1}{n}}$, and
    \begin{eqnarray}
      \widetilde{F}^{ij}=\frac{\partial \widetilde{F}}{\partial u_{ij}}=\frac{1}{n}
      \det(W)^{\frac{1-n}{n}}\frac{\partial \det(W)}{\partial    w_{\overline{\alpha}\overline{\beta}}}\frac{\partial w_{\overline{\alpha}\overline{\beta}}}{\partial u_{ij}},
    \end{eqnarray}
    and
    \begin{eqnarray}
      \widetilde{F}^{pq,rs}&=&\frac{\partial^{2} \widetilde{F}}{\partial u_{pq}\partial u_{rs}}\nonumber\\
      &=&\frac{1}{n}\det(W)^{\frac{1-n}{n}}
      \frac{\partial^{2} \det(W)}{\partial w_{\overline{\alpha}\overline{\beta}}\partial w_{\overline{\eta}\overline{\xi}}}\frac{\partial w_{\overline{\alpha}\overline{\beta}}}{\partial u_{pq}}\frac{\partial w_{\overline{\eta}\overline{\xi}}}{\partial u_{rs}},
    \end{eqnarray}
    since $w_{\overline{\alpha}\overline{\beta}}$ is a linear combination of $u_{ij},\ 1\leq i,j\leq n$.
 Differentiating the equation (\ref{eq3}) twice, we have
 \begin{eqnarray}
   \widetilde{F}^{ij}u_{ijl}=D_l\widetilde{f},\label{4.7}
 \end{eqnarray}
 and
 \begin{eqnarray}
   \widetilde{F}^{pq,rs}u_{pq\xi}u_{rs\xi}+\widetilde{F}^{ij}u_{ij\xi\xi}
   =D_{\xi\xi}\widetilde{f}.\label{4.8}
 \end{eqnarray}
 By the concavity of $\det(W)^{\frac{1}{n}}$ operator with respect to $W$, we have
    \begin{eqnarray}
      D_{\xi\xi}\widetilde{f}=
      \widetilde{F}^{pq,rs}u_{pq\xi}u_{rs\xi}+\widetilde{F}^{ij}u_{ij\xi\xi}
      \leq\widetilde{F}^{ij}u_{ij\xi\xi}.\label{4.9}
    \end{eqnarray}

 Now we contract (\ref{4.6}) with $\widetilde{F}^{ij}$ to get, using
 (\ref{4.7})-(\ref{4.9}),
 \begin{eqnarray}
      \widetilde{F}^{ij}v_{ij}&=&\widetilde{F}^{ij}u_{ij\xi\xi}-\widetilde{F}^{ij}D_{ij}a^{l}u_{l}-2\widetilde{F}^{ij}D_{i}a^{l}u_{lj}
      -\widetilde{F}^{ij}u_{ijl}a^{l}\nonumber\\
      &&-\widetilde{F}^{ij}D_{ij}b+2K_{1}\widetilde{\mathcal{F}}+2K_{2}\widetilde{F}^{ij}u_{il}u_{jl}+2K_{2}\widetilde{F}_{ij}u_{ijl}u_{l}\nonumber\\
      &\geq&D_{\xi\xi}\widetilde{f}-\widetilde{F}^{ij}D_{ij}a^{l}u_{l}-2\widetilde{F}^{ij}D_{i}a^{l}u_{ij}-a^{l}D_{l}\widetilde{f}-\widetilde{F}^{ij}D_{ij}b\nonumber\\
      &&+2K_{1}\widetilde{\mathcal{F}}+2K_{2}\widetilde{F}^{ij}u_{il}u_{jl}+2K_{2}u_{l}D_{l}\widetilde{f}.
    \end{eqnarray}
 where $\widetilde{\mathcal{F}}=\sum\limits_{i=1}^{n}\widetilde{F}^{ii}$.
 Note that
    \begin{eqnarray}
      D_{\xi\xi}\widetilde{f}=\widetilde{f}_{\xi\xi}+2\widetilde{f}_{\xi z}u_{\xi}+\widetilde{f}_{z}u_{\xi\xi},\nonumber\\
      D_{ij}a^{l}=2(\xi\cdot\nu)\xi'^{l}\phi_{zz}u_{ij}+r^{l}_{ij},\nonumber\\
      D_{ij}b=2(\xi\cdot\nu)\xi'^{l}\phi_{x_{l}z}u_{ij}+r_{ij},\nonumber
    \end{eqnarray}
    with $|r^{l}_{ij}|, |r_{ij}|\leq C(|u|_{C^{1}}, |\phi|_{C^{3}}, |\partial \Omega|_{C^{4}})$. At the maximum point $x_{0}\in\Omega$ of $v$, we can assume $(u_{ij})_{n\times n}$ is diagonal. It follows that, by the Cauchy-Schwartz inequality,
 \begin{eqnarray}\label{4.10}
      \widetilde{F}^{ij}v_{ij}&\geq&-C(\mathcal{\widetilde{F}}+K_2+1)-C\widetilde{F}^{ii}|u_{ii}|+\widetilde{f}_{z}u_{\xi\xi}\nonumber\\
      &&+2K_{1}\widetilde{\mathcal{F}}+2K_{2}\widetilde{F}^{ii}u_{ii}^2\nonumber\\
      &\geq&-C(\mathcal{\widetilde{F}}+K_2+1)+\widetilde{f}_{z}u_{\xi\xi}\nonumber\\
      &&+2K_{1}\widetilde{\mathcal{F}}+(2K_{2}-1)\widetilde{F}^{ii}u_{ii}^2,
    \end{eqnarray}
    where $C=C(|u|_{C^{1}}, |\phi|_{C^{3}}, |\partial \Omega|_{C^{4}}, |f|_{C^{2}})$.

 Assume $u_{11}\geq u_{22}\cdots\geq u_{nn}$, and denote $\lambda_{1}\geq\lambda_{2}\geq\cdots\geq\lambda_{n}$ the eigenvalues of the matrix $(w_{\overline{\alpha}\overline{\beta}})_{n\times n}$.
It is easy to see    $\lambda_{1}=u_{11}+\sum\limits_{i=2}^{n-1}u_{ii}\leq (n-1)u_{11}$. Then we have, by (\ref{sw1}) in Proposition \ref{pro2} and (\ref{2.8}) in Proposition \ref{pro3},
  \begin{eqnarray}
    \widetilde{F}^{11}u_{11}^{2}&=&\sum_{1\in\overline{\alpha}}
    \frac{1}{n}\det(W)^{\frac{1-n}{n}} S_{n-1}(\lambda|N_{\overline{\alpha}})u_{11}^{2}\nonumber\\
    &\geq&\frac{1}{(n-1)n}\det(W)^{\frac{1-n}{n}} S_{n-1}(\lambda|1)\lambda_{1}u_{11}\nonumber\\
    &=&\frac{1}{(n-1)n}\det(W)^{\frac{1}{n}}u_{11}=\frac{\widetilde{f}}{(n-1)n}
    u_{11}.\label{4.11}
  \end{eqnarray}
 We can assume $u_{\xi\xi}\geq 0$, otherwise we have (\ref{4.3}). Plug (\ref{4.11}) into (\ref{4.10}) and use the Cauchy-Schwartz inequality, then
\begin{eqnarray}
  \widetilde{F}^{ii}v_{ii}&\geq&(K_{2}-1)\sum_{i=1}^{n}\widetilde{F}^{ii}u_{ii}^{2}
  +(\frac{K_{2}\widetilde{f}}{(n-1)n}+\widetilde{f}_{z})u_{\xi\xi}
  \\&&+(2K_{1}-C)\mathcal{\widetilde{F}}-C(K_2+1).\nonumber
\end{eqnarray}
Choose $K_{2}=\frac{(n-1)\max |f_{z}|}{\min f}+1$ and $K_{1}=C(K_2+2)+1$. It follows that
\begin{eqnarray}
    \widetilde{F}^{ii}v_{ii}&\geq&(2K_{1}-C)\widetilde{\mathcal{F}}-C(K_2+1)>0,
\end{eqnarray}
since we have $\widetilde{\mathcal{F}}\geq 1$ from (\ref{2.10}). This implies that $v(x,\xi)$ attains its maximum on the boundary by the maximum principle. Now we assume $(x_{0},\xi_{0})\in \partial\Omega\times \mathbb{S}^{n-1}$ is the maximum pint of $v(x,\xi)$ in $\overline{\Omega}\times\mathbb{S}^{n-1}$.  Then we consider two cases as follows,

$\mathbf{Case1}$. $\xi_{0}$ is a tangential vector at $x_{0}\in\partial\Omega$.

We directly have $\xi_0\cdot\nu=0$ , $\nu=-Dd$, $v'(x_0,\xi_0)=0$, and $u_{\xi_0, \xi_0}(x_0)>0$. As in \cite{l}, we define
\begin{eqnarray}
  c^{ij}=\delta_{ij}-\nu^i\nu^j,\quad\ \text{in}\ \Omega_{\mu},
\end{eqnarray}
and it is easy to see that $c^{ij}D_j$ is a tangential direction  on $\partial\Omega$.
We compute at $(x_0, \xi_0)$.

From the boundary condition, we have
\begin{eqnarray}
  u_{li}\nu^l&=&(c^{ij}+\nu^i\nu^j)\nu^lu_{lj}\nonumber\\
  &=&c^{ij}u_j\phi_z+c^{ij}\phi_{x_j}-c^{ij}u_lD_j\nu^l+\nu^i\nu^j\nu^lu_{lj}.\label{4.12}
\end{eqnarray}
It follows that
\begin{eqnarray}
  u_{lip}\nu^l&=&[c^{pq}+\nu^p\nu^q]u_{liq}\nu^l\nonumber\\
  &=&c^{pq}D_q(c^{ij}u_j\phi_z+c^{ij}\phi_{x_j}-c^{ij}u_lD_j\nu^l+\nu^i\nu^j\nu^lu_{lj})
  -c^{pq}u_{li}D_q\nu^l+\nu^p\nu^q\nu^lu_{liq},\nonumber
\end{eqnarray}
then we obtain
\begin{eqnarray}
  u_{\xi_0\xi_0\nu}&=&\sum_{ilp=1}^{n}\xi_0^i\xi_0^pu_{lip}\nu^l\nonumber\\
  &=&\sum_{i=1}^{n}\xi_0^i\xi_0^q[D_q(c^{ij}u_j\phi_z+c^{ij}\phi_{x_j}-c^{ij}u_lD_j\nu^l+\nu^i\nu^j\nu^lu_{lj})-u_{li}D_q\nu^l]\nonumber\\
    &\leq& -2\xi_0^i\xi_0^q u_{li}D_q\nu^l+C(1+|u_{\nu\nu}|).
\end{eqnarray}
We use $\phi_z\leq0$ in the last inequality.
We assume $\xi_0=e_1$, it is easy to get the bound for $u_{1i}(x_0)$ for $i>1$ from the maximum of $v(x,\xi)$ in the $\xi_0$ direction. In fact, we can assume $\xi(t)=\frac{(1, t, 0,\cdots, 0)}{\sqrt{1+t^2}}$. Then we have
\begin{eqnarray}
  0&=&\frac{dv(x_0,\xi(t))}{dt}|_{t=0}\nonumber\\
  &=&2u_{12}(x_0)-2\nu^2(\phi_zu_1-u_lD_l\nu^l),\nonumber
\end{eqnarray}
so
\begin{eqnarray}
  |u_{12}|(x_0)\leq C+C|Du|.\label{4.13}
\end{eqnarray}
Similarly, we have for $\forall i>1$,
\begin{eqnarray}
  |u_{1i}|(x_0)\leq C+C|Du|.\label{4.14}
\end{eqnarray}
Thus we have, by $D_1\nu^1\geq\kappa_{min}>0$,
\begin{eqnarray}
  u_{\xi_0\xi_0\nu}&\leq&-2D_1\nu^1u_{11}+C(1+|u_{\nu\nu}|)\nonumber\\
  &\leq&-2\kappa_{min} u_{\xi_0\xi_0}+C(1+|u_{\nu\nu}|).\nonumber
\end{eqnarray}
On the other hand, we have from the Hopf lemma, (\ref{4.5}) and (\ref{4.14}),
\begin{eqnarray}
  0&\leq&v_{\nu}(x_0,\xi_0)\nonumber\\
  &=&u_{\xi_0\xi_0 \nu}-D_{\nu}a^{l}u_{l}-a^{l}u_{\nu\nu}-D_{\nu}b+2K_{1}x_{i}\nu^i+2K_{2}u_{l}u_{l\nu}\nonumber\\
    &\leq&-2\kappa_{min}u_{\xi_0\xi_0}+C(1+|u_{\nu\nu}|).\nonumber
\end{eqnarray}
Then we get,
\begin{eqnarray}
  u_{\xi_0\xi_0}(x_0)\leq C(1+|u_{\nu\nu}|).
\end{eqnarray}

   \textbf{Case2.} $\xi_0$ is non-tangential.

   We can find a tangential vector $\tau$, such that $\xi_0 = \alpha\tau+\beta\nu$,
with $\alpha^2 + \beta^2 = 1$. Then we have
\begin{eqnarray}
  u_{\xi_0\xi_0}(x_0)&=&\alpha^2u_{\tau\tau}(x_0)+\beta^2u_{\nu\nu}(x_0)+2\alpha\beta u_{\tau\nu}(x_0)\nonumber\\
  &=&\alpha^2u_{\tau\tau}(x_0)+\beta^2u_{\nu\nu}(x_0)+2(\xi_0\cdot\nu)\xi'_0\cdot(\phi_zDu-u_lD\nu^l).\nonumber
\end{eqnarray}
By the definition of $v(x_0,\xi_0)$,
\begin{eqnarray}
  v(x_0,\xi_0)&=&\alpha^2v(x_0,\tau)+\beta^2v(x_0,\nu)\nonumber\\
  &\leq&\alpha^2v(x_0,\xi_0)+\beta^2v(x_0,\nu).\nonumber
\end{eqnarray}
Thus,
\begin{eqnarray}
  v(x_0,\xi_0)=v(x_0,\nu),\nonumber
\end{eqnarray}
and
\begin{eqnarray}
  u_{\xi_0\xi_0}(x_0)\leq |u_{\nu\nu}|+C.
\end{eqnarray}
In conclusion, we have (\ref{4.3}) in both cases.
 \end{proof}

First, we denote $d(x)=dist(x,\partial\Omega)$, and define
\begin{eqnarray}
  h(x)=-d(x)+K_3d^{2}(x).\label{h1}
\end{eqnarray}
where $K_3$ is large constant to be determined later. Then we give the following key Lemma.

\begin{lemma}\label{le4.2}
  Suppose $\Omega\subset\mathbb{R}^n$ is a bounded strictly
  convex domain  with $C^2$ boundary.  Denote $\kappa_{max}(x)$ ($\kappa_{min}(x)$) the maximum (minimum) principal curvature at $x\in\partial\Omega$. Let  $u\in C^2(\overline{\Omega})$ is  strictly $(n-1)$-convex
   and $h(x)$ is defined as in (\ref{h1}).
 Then, for any $\gamma\in[\frac{1}{2},1)$, there exists $K_3$, a sufficiently large number depends only on $n$, $m$, $k$, $\gamma$, $\min f$ and  $\Omega$, such that,
  \begin{eqnarray}
    F^{ij}h_{ij}\geq \gamma\kappa_0(1+\mathcal{F}),\quad\text{in}\ \Omega_{\mu}\ (0<\mu\leq\widetilde{\mu}),\label{h3}
  \end{eqnarray}
 where $\kappa_0=\frac{H}{n-1}\geq\kappa_{min}$ and
 $\widetilde{\mu}=\min\{\frac{1}{4K_3},\frac{2-\gamma}{2K_3},\frac{1}{2\kappa_{min}},
 \mu_0\}$, $\mu_0$ is mentioned in (\ref{d}).
 As $\gamma$ tends to 1, $K_3$ tends to infinity.
\end{lemma}

\begin{proof}
   For $x_0\in\Omega_{\mu}$, there exists $y_0\in\partial\Omega$ such that $|x_0-y_0|=d(x_0)$. Then, in terms of a principal coordinate system at $y_0$, we have (see \cite{gt}, Lemma 14.17),
  \begin{eqnarray}
    [D^2d(x_0)]=-diag\big[\frac{\kappa_1}{1-\kappa_1d},\cdots,\frac{\kappa_{n-1}}{1-\kappa_{n-1}d},0\big],\label{d2}
  \end{eqnarray}
  and
  \begin{eqnarray}
    Dd(x_0)=-\nu(x_0)=(0,\cdots,0,-1).
  \end{eqnarray}
  Observe that
  \begin{eqnarray}
    [D^2h(x_0)]=diag\big[\frac{((1-2K_3d)\kappa_1}{1-\kappa_1d},\cdots,\frac{(1-2K_3d)
    \kappa_{n-1}}{1-\kappa_{n-1}d},2K_3\big].\label{h2}
  \end{eqnarray}
  Denote $\mu_i=\frac{(1-K_3d)\kappa_i}{1-\kappa_id}>0, \ \forall 1\leq i\leq n-1$, and $\mu_n=2K_3$ for simplicity. Then we define $\lambda(D^2h)=\{\mu_{i_1}+\cdots+\mu_{i_{n-1}}|\ 1\leq i_1<\cdots<i_{n-1}\leq n\}$ and assume $\lambda_1\geq\cdots\geq\lambda_{n-1}\geq\lambda_n$, it is easy to see
  \begin{eqnarray}
    \lambda_{n-1}\geq 2K_3+\sum_{l=1}^{m-1}\mu_{i_l}\geq K_3,\label{4.15}
  \end{eqnarray}
  if we choose $K_3$ sufficiently large and $\mu\leq\frac{1}{4K_3}$. It is also
  easy to see
  that $h$ is strictly convex.

  We now consider the function $w=h-\frac{1}{2}\gamma\kappa_0|x|^2$. As above, we
   define $\widetilde{\mu}(D^2w)=(\widetilde{\mu}_1,\cdots,\widetilde{\mu}_n)$
   the eigenvalues of the Hessian $D^2w$, and $\widetilde{\lambda}=\{\widetilde{\mu}_{i_1}+\cdots+\widetilde{\mu}_{i_{n-1}}|\
   1\leq i_1<\cdots<i_{n-1}\leq n\}$ with $\widetilde{\lambda}_1\geq\cdots\geq\widetilde{\lambda}_n$. For any $\gamma\in
   [\frac{1}{2},1)$, assume
   $\mu\leq\min\{\frac{1}{4K_3},\frac{2-\gamma}{2K_3},\frac{1}{2\kappa_{min}}\}$,
   we have
   \begin{eqnarray}
    \frac{1-K_3d}{1-\kappa_id}
     >\gamma,\quad\forall i=1,2,\cdots,n-1.\nonumber
   \end{eqnarray}
    Set $\delta=\frac{1}{2}(\frac{1-K_3d}{1-\kappa_{min}d}-\gamma)$ independent of $K_3$,
    recalling $H=\sum\limits_{i=1}^{n-1}\kappa_i$,
     it follows that
   \begin{eqnarray}
     \widetilde{\lambda}_{n}&=&\sum_{i=1}^{n-1}\mu_i-(n-1)\gamma\kappa_0\nonumber\\
     &\geq&(n-1)\delta\kappa_0.\label{4.16}
   \end{eqnarray}

  By the concavity of $\widetilde{F}$, we have
  \begin{eqnarray}
    \widetilde{F}^{ij}w_{ij}&\geq&\widetilde{F}[D^2u+D^2w]-\widetilde{F}[D^2u]\nonumber\\
    &\geq&\widetilde{F}[D^2w]\nonumber\\
    &\geq&K_3^{n-1}((n-1)\delta\kappa_0)\nonumber\\
    &\geq& K_3,
  \end{eqnarray}
  for a large enough $K_3\geq\frac{1}{(n-1)\delta\kappa_0}$.
  Then we get
  \begin{eqnarray}
    \widetilde{F}^{ij}h_{ij}=\widetilde{F}^{ij}(h-\frac{1}{2}\gamma\kappa_0|x|^2
    +\frac{1}{2}\gamma\kappa_0|x|^2)_{ij}\geq K_3+\gamma\kappa_0
    \widetilde{\mathcal{F}}.
  \end{eqnarray}
  If we choose $K_3\geq \frac{\gamma\kappa_0\max f^{\frac{1}{n}}}{n\min f}$,
  then we have
  \begin{eqnarray}
    F^{ij}h_{ij}\geq \gamma\kappa_0(1+\mathcal{F}).
  \end{eqnarray}
\end{proof}

Following the line of Qiu-Ma \cite{mq} and Chen-Zhang \cite{cz},
 we construct the sub barrier function as
\begin{eqnarray}
  P(x)=g(x)(Du\cdot \nu-\phi(x,u))-G(x).
\end{eqnarray}
with
\begin{eqnarray}
\nu(x)&=&-Dd(x),\nonumber\\
  g(x)&=&1-\beta h(x),\nonumber\\
  G(x)&=&(A+\sigma N)h(x),\nonumber
\end{eqnarray}
where$A$, $\sigma$,
 and $\beta$ are positive constants to be determined.
We have the following lemma.
\begin{lemma}\label{le4.3}
  Fix $\sigma$, if we select $\beta$ large, $\mu$ small, $A$ large,
  and assume $N$ large, then
  \begin{eqnarray}
    P\geq0,\quad\text{in}\quad\Omega_{\mu}.
  \end{eqnarray}
  Furthermore, we have
  \begin{eqnarray}\label{4.21}
    \sup_{\partial \Omega}u_{\nu\nu}\leq C+\sigma N,
  \end{eqnarray}
  where constant $C$ depends  only on $|u|_{C^{1}}$, $|\partial\Omega|_{C^{2}}$ $|f|_{C^2}$ and $|\phi|_{C^{2}}$.
\end{lemma}
\begin{proof}
  We assume $P(x)$ attains its minimum point $x_{0}$ in the interior of $\Omega_{\mu}$. Differentiate $P$ twice to  obtain
  \begin{eqnarray}
    P_{i}=g_{i}(u_{l}\nu^{l}-\phi)+g(u_{li}\nu^{l}+u_{l}D_i\nu^l-D_i\phi)-G_{i},
  \end{eqnarray}
  and
  \begin{eqnarray}
    P_{ij}&=&g_{ij}(u_{l}\nu^{l}+\phi)+g_{i}(u_{lj}\nu_{l}+u_{l}D_j\nu^l-D_j\phi)\\
    &&+g_{j}(u_{li}\nu^{l}+u_{l}D_i\nu^l-D_i\phi)+g(u_{lij}\nu^{l}+u_{li}D_j\nu^l\nonumber\\
    &&+u_{lj}D_i\nu^l+u_{l}D_{ij}\nu^l-D_{ij}\phi)-G_{ij}.\nonumber
  \end{eqnarray}

  By a rotation of coordinates, we may assume that $(u_{ij})_{n\times n}$ is diagonal at $x_{0}$, so are $W$ and $(F^{ij})_{n\times n}$.

  We choose $\mu<\min\{\widetilde{\mu},\frac{2\epsilon}{\beta},\frac{\epsilon}{2K_3}\}$, where $\widetilde{\mu}$ is defined in Lemma \ref{le4.2} and $\epsilon\in(0,\frac{1}{2})$
  is a small positive number to be determined
   , such that $|\beta h|\leq \beta\frac{\mu}{2}\leq\epsilon$. It follows that
  \begin{eqnarray}\label{4.17}
    1\leq g\leq1+\epsilon.
  \end{eqnarray}
  Remember that $h_i=-(1-2K_3d)d_i$, we also have
  \begin{eqnarray}
    (1-\epsilon)|d_i|\leq|h_i|\leq|d_i|.\label{4.19}
  \end{eqnarray}

  By a straight computation, using Lemma \ref{le4.2}, we obtain
  \begin{eqnarray}\label{4.18}
    F^{ij}P_{ij}&=&F^{ii}g_{ii}(u_{l}\nu^{l}-\phi)+2F^{ii}g_{i}(u_{ii}\nu^{i}+u_{l}D_i\nu^l
    -D_i\phi)\nonumber\\
     &&+g F^{ii}(u_{lii}\nu^{l}+2u_{ii}D_i\nu^{i}+u_{l}D_{ii}\nu^l-D_{ii}\phi)-(A+\sigma N)F^{ii}h_{ii} \nonumber\\
    &\leq&\big( \beta C_1-(A+\sigma N)\gamma\kappa_0\big)(\mathcal{F}+1)\\
   && -2\beta F^{ii}u_{ii}h_{i}\nu^i+2gF^{ii}u_{ii}D_i\nu^{i},\nonumber
  \end{eqnarray}
where $C_1=C_{1}(|u|_{C^{1}}, |\partial \Omega|_{C^{3}}, |\phi|_{C^{2}}, |f|_{C^{1}}, n)$.

  We divide indexes $I=\{1, 2, \cdots, n\}$ into two sets in the following way,
  \begin{eqnarray}
    B=\{i\in I | |\beta d_{i}^{2}|<\epsilon\kappa_{min}\},\nonumber\\
    G =I\backslash B =\{i\in I | |\beta d_{i}^{2}|\geq\epsilon\kappa_{min}\},\nonumber
  \end{eqnarray}
  where $\kappa_{min}$ ($\kappa_{max}$) is the minimum (maximum)
  principal curvature of the boundary.
  For $i\in G$, by $P_{i}(x_{0})=0$, we get
  \begin{eqnarray}
    u_{ii}=(1-2K_3d)[\frac{(A+\sigma N)}{g}+\frac{\beta(u_{l}\nu^l-\phi)}
    {g}]+\frac{u_{l}D_i\nu^l-D_i\phi}{d_{i}}.
  \end{eqnarray}
  Because $|d_{i}^{2}|\geq\frac{\epsilon\kappa_{min}}{\beta}$, (\ref{4.17})
  and (\ref{4.19}), we have
  \begin{eqnarray}
    |\frac{(1-2K_3d)\beta(u_l\nu^l-\phi)}{g}+\frac{u_{l}D_i\nu^l-D_i\phi}{d_{i}}|\leq\beta C_{2}(\epsilon^{-1}, |u|_{C^{1}},
    |\partial \Omega|_{C^{2}}, |\psi|_{C^{1}}).\nonumber
  \end{eqnarray}
  Then let $A\geq3\beta C_{2}$, we have
  \begin{eqnarray}
    \frac{A}{3}+\frac{1-\epsilon}{1+\epsilon}\sigma N
    \leq u_{ii}\leq \frac{4A}{3}+\sigma N,\label{4.20}
  \end{eqnarray}
  for $\forall i\in G$. We choose $\beta\geq2n\epsilon\kappa_{min}+1$ to let $|d_{i}^{2}|\leq\frac{1}{2n}$ for $i\in B$. Because $|Dd|=1$, there is a $i_{0}\in G$, say $i_{0}=1$, such that
  \begin{eqnarray}
    d_{1}^{2}\geq \frac{1}{n}.\label{4.22}
  \end{eqnarray}

  We have
  \begin{eqnarray}\label{4.30}
    -2\beta \sum_{i\in I}F^{ii}u_{ii}h_{i}\nu^i&=&-2\beta\sum_{i\in G}F^{ii}u_{ii}h_{i}\nu^i-2\beta \sum_{i\in B}F^{ii}u_{ii}h_{i}\nu^i\\
    &\leq& -2(1-\epsilon)\beta F^{11}u_{11}d_{1}^{2}-2\beta\sum_{i\in B,
    u_{ii}<0}F^{ii}u_{ii}d_{i}^{2}\nonumber\\
    &\leq&-\frac{\beta F^{11}u_{11}}{n}-2\epsilon
    \kappa_{min}\sum_{u_{ii}<0}F^{ii}u_{ii}.\nonumber
  \end{eqnarray}
  and
  \begin{eqnarray}\label{4.31}
    2g\sum_{i\in I}F^{ii}u_{ii}D_i\nu^i&=&2g\sum_{u_{ii}\geq 0}F^{ii}u_{ii}D_i\nu^i+2g\sum_{u_{ii}<0}F^{ii}u_{ii}D_i\nu^i\\
    &\leq&2\kappa_{max}\sum_{u_{ii}\geq0}F^{ii}u_{ii}
    +2\kappa_{min}\sum_{u_{ii}<0}F^{ii}u_{ii}.\nonumber
  \end{eqnarray}
  Plug (\ref{4.30}) and (\ref{4.31}) into (\ref{4.18}) to get
  \begin{eqnarray}\label{4.32}
    F^{ii}P_{ij}&\leq&\big(\beta C_{1}-(A+\sigma N)\gamma\kappa_0\big)(\mathcal{F}+1)-\frac{\beta}{2n}F^{11}u_{11}\nonumber\\
    &&+2(1-\epsilon)\kappa_{min}\sum_{u_{ii}<0}F^{ii}u_{ii}
    +2\kappa_{max}\sum_{u_{ii}\geq0}F^{ii}u_{ii}.
  \end{eqnarray}

  Denote $u_{22}\geq\cdots\geq u_{nn}$, and
  \begin{eqnarray}
    \lambda_{1}&=&\max\limits_{1\in\overline{\alpha}}\{w_{\overline{\alpha}\overline{\alpha}}\}
   =\mu_1+\sum_{i=2}^{n-1}\mu_i,\nonumber\\
    \lambda_{m_1}&=&\min\limits_{1\in\overline{\alpha}}\{w_{\overline{\alpha}\overline{\alpha}}\}=
     u_{11}+\sum_{i=3}^{n}u_{ii}.\nonumber
  \end{eqnarray}
   and $\lambda_{2}\geq\cdots\geq\lambda_{n}>0$ the eigenvalues of  the matrix $W$.
   Assume $N>1$, from (\ref{4.1}) we see that
  \begin{eqnarray}\label{4.34}
    u_{ii}\leq 2C_{0}N,\quad \forall i\in I.
  \end{eqnarray}
  Then
  \begin{eqnarray}\label{4.35}
    \lambda_{i}\leq2(n-1)C_{0}N,\quad \forall 1\leq i\leq C_{n}^{m}.
  \end{eqnarray}

\medskip

  If $u_{11}\leq u_{22}$, we see that $\lambda_{m_1}=\lambda_n$.
  Then
  \begin{eqnarray}
    F^{11}>S_{n-1}(\lambda|n)\geq\frac{1}{n(n-1)}\mathcal{F},
  \end{eqnarray}
  it follows that
  \begin{eqnarray}
    F^{ij}P_{ij}&\leq&\big(\beta C_{1}-(A+\sigma N)\gamma\kappa_0\big)(\mathcal{F}+1)
    +2C_{0}\kappa_{max}N\mathcal{F}\nonumber\\
    &&-\frac{\beta}{2n^2(n-1)}(\frac{A}{3}+\frac{1-\epsilon}{1+\epsilon}\sigma N)\mathcal{F}\nonumber\\
    &<&0.
  \end{eqnarray}
  if we choose $\beta>\frac{12n^2(n-1)\kappa_{max}C_{0}}{\sigma}$ and $A>\frac{\beta C_{1}}{\gamma\kappa_0}$.

  In the following cases, we always assume $u_{11}>u_{22}$.

\medskip

  $\mathbf{Case 1}$. $u_{nn}\geq0$.

  It follows from
  \begin{eqnarray}
    kf=\sum_{i=1}^{n}F^{ii}u_{ii}=\sum_{u_{ii}\geq0}F^{ii}u_{ii}\nonumber
  \end{eqnarray}
  and (\ref{4.32}) that
  \begin{eqnarray}
    F^{ij}P_{ij}\leq \big(\beta C_{1}-(A+\sigma N)
    \gamma\kappa_0\big)(\mathcal{F}+1)+2\kappa_{max}kf<0,
  \end{eqnarray}
  if we choose $A>\frac{\beta C_{1}+2\kappa_{max}k\max f}{\gamma\kappa_0}$.

\medskip

  $\mathbf{Case 2}$.  $B=\sum\limits_{u_{ii}<0}u_{ii}>-(n-2)\sigma N-\epsilon N$
   and
   $\lambda_n\leq\epsilon N$.

  It follows from
  \begin{eqnarray*}
    \lambda_n=\sum_{i=2}^nu_{ii},
  \end{eqnarray*}that
  \begin{eqnarray}
   && 2\kappa_{max}\sum_{u_{ii}\geq0}F^{ii}u_{ii}
    +(2-\epsilon)\kappa_{min}\sum_{u_{ii}<0}F^{ii}u_{ii}\nonumber\\
    &\leq& 2\kappa_{max}F^{11}u_{11}+
    2[\kappa_{max}(\epsilon N-B)
    +(1-\epsilon)\kappa_{min}B]\mathcal{F}
    \nonumber\\&\leq&
    2(n-2)
    [\kappa_{max}-(1-\epsilon)\kappa_{min}]\sigma N\mathcal{F}+
    4\epsilon\kappa_{max} N\mathcal{F}\nonumber\\&&
  + 2\kappa_{max}F^{11}u_{11}.
  \end{eqnarray}
  Since $\kappa_{max}-\kappa_{min}<\frac{\gamma H}{2(n-1)(n-2)}$, we have
  \begin{eqnarray}
    (n-1)\gamma\kappa_0=\gamma H\kappa_{min})>2(n-1)(n-2)
    (\kappa_{max}-\kappa_{min}).
  \end{eqnarray}
  We can choose a sufficiently small $\epsilon=\epsilon(n,\gamma,\kappa_{max},\kappa_{min})$
  to get
  \begin{eqnarray}
     2\kappa_{max}\sum_{u_{ii}\geq0}F^{ii}u_{ii}
    +(2-\epsilon)\kappa_{min}\sum_{u_{ii}<0}F^{ii}u_{ii}
    \leq\gamma\kappa_0\sigma N\mathcal{F}+2\kappa_{max}F^{11}u_{11}.\nonumber
  \end{eqnarray}
 We now choose $A>\beta C_{1}+1$ and $\beta\geq 4n\kappa_{max}$ to get
  \begin{eqnarray}
    F^{ij}P_{ij}<0.
  \end{eqnarray}

\medskip

  $\mathbf{Case 3}$. $B=\sum\limits_{u_{ii}<0}u_{ii}>-(n-2)\sigma N-\epsilon N$ and
   $\lambda_n>\epsilon N$.

It is easy to see, by (\ref{4.35}), that,
\begin{eqnarray}
  F^{11}&>&S_{n-1}(\lambda|1)=\lambda_2\cdots\lambda_n\nonumber\\
  &\geq&\epsilon^{n-1}N^{n-1}=(\frac{\epsilon}{2(n-1)C_0})^{n-1}[2(n-1)C_0N]^{n-1}\nonumber\\
  &\geq&\frac{1}{n}(\frac{\epsilon}{2(n-1)C_0})^{n-1}S_{n-1}(\lambda).
\end{eqnarray}
Similarly, if we choose $\beta>\frac{2^{n+1}3n^2(n-1)^{n}\kappa_{max}(C_0)^n}{\sigma
\epsilon^{n-1}}$ and $A>\frac{\beta C_1}{\gamma\kappa_0}$, then
\begin{eqnarray}
  F^{ij}P_{ij}<0.
\end{eqnarray}

\medskip

\textbf{Case4.} $B=\sum\limits_{u_{ii}<0}u_{ii}\leq-(n-2)\sigma N-\epsilon N$.

We have
\begin{eqnarray}
  \lambda_n=u_{22}+\sum_{i=3}^nu_{ii}>0.\nonumber
\end{eqnarray}
It follows that
\begin{eqnarray}
  u_{22}\geq\frac{|B|}{n-2}\geq(\sigma +\frac{\epsilon}{n-2})N>u_{11},\nonumber
\end{eqnarray}
if we assume $N>\frac{4(n-2)A}{3\epsilon}$. This contradicts to that $u_{11}>u_{22}$.

In conclusion, we choose a small $\epsilon=\epsilon(n,\gamma,\kappa_{max},\kappa_{min})$,
\begin{eqnarray}
  \beta=\max\{4n\kappa_{max}+1,\frac{2^{n+1}3n^2(n-1)^{n}\kappa_{max}(C_0)^n}{\sigma
\epsilon^{n-1}}\}.\nonumber
\end{eqnarray}
 and $\mu=\min\{\widetilde{\mu},\frac{2\epsilon}{\beta},\frac{\epsilon}{2K_3}\}$.
 If $A>\max \{3\beta C_{2},\frac{\beta C_{1}+2\kappa_{max}k\max f}{\gamma\kappa_0}\}$
 and $N>\frac{4(n-2)A}{3\epsilon}$,
 we obtain $F^{ii}P_{ij}<0$, which contradicts to
  that $P$ attains its minimum in the interior of $\Omega_{\mu}$.
   This implies that $P$ attains its minimum on the boundary $\partial \Omega_{\mu}$.

On $\partial \Omega$, it is easy to see
\begin{eqnarray}
  P=0.
\end{eqnarray}
On $\partial\Omega_{\mu}\cap\Omega$, we have
\begin{eqnarray}
  P\geq-C_{3}(|u|_{C^{1}}, |\phi|_{C^{0}})+(A+\sigma N)\frac{\mu}{2}\geq0,
\end{eqnarray}
if we take $A=\max \{\frac{2C_{3}}{\mu},3\beta C_{2},
\frac{\beta C_{1}+2\kappa_{max}k\max f}{\gamma\kappa_0}\}$. Finally the maximum principle tells us that
\begin{eqnarray}
  P\geq0,\quad\text{in}\quad\Omega_{\mu}.
\end{eqnarray}

Suppose $u_{\nu\nu}(y_{0})=\sup_{\partial\Omega}u_{\nu\nu}>0$, we have
\begin{eqnarray}
  0&\geq&P_{\nu}(y_{0})\nonumber\\
  &\geq&(u_{\nu\nu}+u_lD_i\nu^l\nu^i-D_{\nu}\phi)-(A+\sigma N)h_{\nu}\nonumber\\
  &\geq&u_{\nu\nu}(y_{0})-C(|u|_{C^{1}}, |\partial\Omega|_{C^{2}}, |\phi|_{C^{2}})-(A+\sigma N).\nonumber
\end{eqnarray}
Then we get
\begin{eqnarray}
  \sup_{\partial\Omega}u_{\nu\nu}\leq C+\sigma N.
\end{eqnarray}
\end{proof}

 In a similar way, we construct the super barrier function as
 \begin{eqnarray}
  \overline{P}(x):=g(x)(Du\cdot \nu-\phi(x))+G(x).
\end{eqnarray}
We also have the following lemma.
\begin{lemma}\label{le4.4}
 Fix $\sigma$, if we select $\beta$ large, $\mu$ small, $A$ large,
   then
  \begin{eqnarray}
    \overline{P}\leq0,\quad\text{in}\quad\Omega_{\mu}.
  \end{eqnarray}
  Furthermore, we have
  \begin{eqnarray}\label{4.36}
    \inf_{\partial \Omega}u_{\nu\nu}\geq -C-\sigma N,
  \end{eqnarray}
  where constant $C$ depends on $|u|_{C^{1}}$, $|\partial\Omega|_{C^{2}}$ $|f|_{C^2}$ and $|\phi|_{C^{2}}$.
\end{lemma}
\begin{proof}
   We assume $\overline{P}(x)$ attains its maximum point $x_{0}$ in the interior of $\Omega_{\mu}$. Differentiate $\overline{P}$ twice to  obtain
  \begin{eqnarray}
    \overline{P}_{i}=g_{i}(u_{l}\nu^{l}-\phi)+g(u_{li}\nu^{l}+u_{l}D_i\nu^l-D_i\phi)+G_{i},
  \end{eqnarray}
  and
  \begin{eqnarray}
   \overline{P}_{ij}&=&g_{ij}(u_{l}\nu^{l}+\phi)+g_{i}(u_{lj}\nu_{l}+u_{l}D_j\nu^l-D_j\phi)\\
    &&+g_{j}(u_{li}\nu^{l}+u_{l}D_i\nu^l-D_i\phi)+g(u_{lij}\nu^{l}+u_{li}D_j\nu^l\nonumber\\
    &&+u_{lj}D_i\nu^l+u_{l}D_{ij}\nu^l-D_{ij}\phi)+G_{ij}.\nonumber
  \end{eqnarray}

  As before we assume that $(u_{ij})$ is diagonal at $x_{0}$, so are $W$ and $(F_{ij})$.

  We choose $\mu=\min\{\widetilde{\mu},\frac{2\epsilon}{\beta},\frac{\epsilon}{2K_3}\}$, where $\widetilde{\mu}$ is defined in Lemma \ref{le4.2} and $\epsilon\in(0,\frac{1}{2})$
  is a small positive number to be determined
   , such that $|\beta h|\leq \beta\frac{\mu}{2}\leq\epsilon$. It follows that
  \begin{eqnarray}\label{4.37}
    1\leq g\leq1+\epsilon.
  \end{eqnarray}
  Remember that $h_i=-(1-2K_3d)d_i$, we also have
  \begin{eqnarray}
    (1-\epsilon)|d_i|\leq|h_i|\leq|d_i|.\label{4.38}
  \end{eqnarray}
  By a straight computation, using Lemma \ref{4.2}, we obtain
  \begin{eqnarray}\label{4.39}
    F^{ij}\overline{P}_{ij}&=&F^{ii}g_{ii}(u_{l}\nu^{l}-\phi)+2F^{ii}g_{i}(u_{ii}\nu^{i}+u_{l}D_i\nu^l
    -D_i\phi)\nonumber\\
     &&+g F^{ii}(u_{lii}\nu^{l}+2u_{ii}D_i\nu^{i}+u_{l}D_{ii}\nu^l-D_{ii}\phi)+(A+\sigma N)F^{ii}h_{ii} \nonumber\\
    &\geq&\big((A+\sigma N)\gamma\kappa_0-\beta C_1\big)(\mathcal{F}+1)\\
   && -2\beta F^{ii}u_{ii}h_{i}\nu^i+2gF^{ii}u_{ii}D_i\nu^{i},\nonumber
  \end{eqnarray}
where $C_1=C_{1}(|u|_{C^{1}}, |\partial \Omega|_{C^{31}}, |\phi|_{C^{2}}, |f|_{C^{1}}, n)$.

  We divide indexes $I=\{1, 2, \cdots, n\}$ into two sets in the following way,
  \begin{eqnarray}
    B=\{i\in I | |\beta d_{i}^{2}|<\epsilon\kappa_{min}\},\nonumber\\
    G =I\backslash B =\{i\in I | |\beta d_{i}^{2}|\geq\epsilon\kappa_{min}\},\nonumber
  \end{eqnarray}
  where $\kappa_{min}$ ($\kappa_{max}$) is the minimum (maximum)
  principal curvature of the boundary.
  For $i\in G$, by $\overline{P}_{i}(x_{0})=0$, we get
  \begin{eqnarray}
    u_{ii}=(1-2K_3d)[\frac{-(A+\sigma N)}{g}+\frac{\beta(u_{l}\nu^l-\phi)}
    {g}]+\frac{u_{l}D_i\nu^l-D_i\phi}{d_{i}}.
  \end{eqnarray}
  Because $|d_{i}^{2}|\geq\frac{\epsilon\kappa_{min}}{\beta}$, by (\ref{4.37})
  and (\ref{4.38}), we have
  \begin{eqnarray}
    |\frac{(1-2K_3d)\beta(u_l\nu^l-\phi)}{g}+\frac{u_{l}D_i\nu^l-D_i\phi}{d_{i}}|\leq\beta C_{2}(\epsilon, |u|_{C^{1}},
    |\partial \Omega|_{C^{2}}, |\psi|_{C^{1}}).\nonumber
  \end{eqnarray}
  Then let $A\geq3\beta C_{2}$, we have
  \begin{eqnarray}
    -\frac{4A}{3}-\sigma N\leq u_{ii}\leq
    -\frac{A}{3}-\frac{1-\epsilon}{1+\epsilon}\sigma N,\label{4.40}
  \end{eqnarray}
  for $\forall i\in G$. We choose $\beta\geq2n\epsilon\kappa_{min}+1$ to let $|d_{i}^{2}|\leq\frac{1}{2n}$ for $i\in B$. Because $|Dd|=1$, there is a $i_{0}\in G$, say $i_{0}=1$, such that
  \begin{eqnarray}
    d_{1}^{2}\geq \frac{1}{n}.\label{4.41}
  \end{eqnarray}

  We have
  \begin{eqnarray}\label{4.42}
    -2\beta \sum_{i\in I}F^{ii}u_{ii}h_{i}\nu^i&=&-2\beta\sum_{i\in G}F^{ii}u_{ii}h_{i}\nu^i-2\beta \sum_{i\in B}F^{ii}u_{ii}h_{i}\nu^i\\
    &\geq& -2(1-\epsilon)\beta F^{11}u_{11}d_{1}^{2}-2\beta\sum_{i\in B,
    u_{ii}>0}F^{ii}u_{ii}d_{i}^{2}\nonumber\\
    &\geq&-\frac{\beta F^{11}u_{11}}{n}-2\epsilon
    \kappa_{min}\sum_{u_{ii}>0}F^{ii}u_{ii}.\nonumber
  \end{eqnarray}
  and
  \begin{eqnarray}\label{4.43}
    2g\sum_{i\in I}F^{ii}u_{ii}D_i\nu^i&=&2g\sum_{u_{ii}> 0}F^{ii}u_{ii}D_i\nu^i+2g\sum_{u_{ii}\leq0}F^{ii}u_{ii}D_i\nu^i\\
    &\geq&2\kappa_{min}\sum_{u_{ii}>0}F^{ii}u_{ii}
    +2\kappa_{max}\sum_{u_{ii}\leq0}F^{ii}u_{ii}.\nonumber
  \end{eqnarray}
  Plug (\ref{4.42}) and (\ref{4.43}) into (\ref{4.39}) to get
  \begin{eqnarray}\label{4.32}
    F^{ii}\overline{P}_{ij}&\geq&\big((A+\sigma N-\beta C_1)\gamma\kappa_0\big)
    (\mathcal{F}+1)
    -\frac{\beta}{2n}F^{11}u_{11}\nonumber\\
    &&+2(1-\epsilon)\kappa_{min}\sum_{u_{ii}>0}F^{ii}u_{ii}
    +2\kappa_{max}\sum_{u_{ii}\leq0}F^{ii}u_{ii}.
  \end{eqnarray}

Denote $u_{22}\geq\cdots\geq u_{nn}$, and
  \begin{eqnarray}
    \lambda_{m_1}&=&\min\limits_{1\in\overline{\alpha}}\{w_{\overline{\alpha}\overline{\alpha}}\}=
     u_{11}+\sum_{i=3}^{n}u_{ii},\nonumber
  \end{eqnarray}
   and $\lambda_1\geq\lambda_{2}\geq\cdots\geq\lambda_{n}>0$ the eigenvalues of  the matrix $W$.
   Assume $N>1$, from (\ref{4.1}) we see that
  \begin{eqnarray}\label{4.34}
    u_{ii}\leq 2C_{0}N,\quad \forall i\in I.
  \end{eqnarray}
  Then
  \begin{eqnarray}\label{4.35}
    \lambda_{i}\leq2(n-1)C_{0}N,\quad \forall 1\leq i\leq C_{n}^{m}.
  \end{eqnarray}

\medskip
Since $u_{11}\leq u_{22}$, we see that $\lambda_{m_1}=\lambda_n$.
  Then
  \begin{eqnarray}
    F^{11}>S_{n-1}(\lambda|n)\geq\frac{1}{n(n-1)}\mathcal{F},
  \end{eqnarray}
  it follows that
  \begin{eqnarray}
    F^{ij}\overline{P}_{ij}&\geq&\big((A+\sigma N)\gamma\kappa_0-\beta C_{1}\big)(\mathcal{F}+1)
    -2C_{0}\kappa_{max}N\mathcal{F}\nonumber\\
    &&+\frac{\beta}{2n^2(n-1)}(\frac{A}{3}+\frac{1-\epsilon}{1+\epsilon}\sigma N)\mathcal{F}\nonumber\\
    &>&0.
  \end{eqnarray}
  if we choose $\beta=\frac{12n^2(n-1)\kappa_{max}C_{0}}{\sigma}+2n\epsilon\kappa_{min}+1$ and $A>\frac{\beta C_{1}}{\gamma\kappa_0}$.
This contradicts to that $\overline{P}$ attains its maximum in the
interior of $\Omega_{\mu}$.
This contradiction implies that $\overline{P}$
 attains its maximum on the boundary $\partial \Omega_{\mu}$.


On $\partial \Omega$, it is easy to see
\begin{eqnarray}
  \overline{P}=0.\nonumber
\end{eqnarray}
On $\partial\Omega_{\mu}\cap\Omega$, we have
\begin{eqnarray}
  \overline{P}\leq C_{3}(|u|_{C^{1}}, |\phi|_{C^{0}})-(A+\sigma N)\frac{\mu}{2}\leq0,\nonumber
\end{eqnarray}
if we take $A=\frac{2C_{3}}{\mu}+\frac{\beta C_{1}}{k_{3}}+1$. Finally the maximum principle tells us that
\begin{eqnarray}
  \overline{P}\leq0,\quad\text{in}\quad\Omega_{\mu}.
\end{eqnarray}

Suppose $u_{\nu\nu}(y_{0})=\inf_{\partial\Omega}u_{\nu\nu}$, we have
\begin{eqnarray}
  0&\leq&P_{\nu}(y_{0})\nonumber\\
  &\leq&(u_{\nu\nu}+u_{l}D_i\nu^l\nu^i-D_{\nu}\phi)+(A+\sigma N)h_{\nu}\nonumber\\
  &\leq&u_{\nu\nu}(y_{0})+C(|u|_{C^{1}}, |\partial\Omega|_{C^{2}}, |\phi|_{C^{2}})+(A+\sigma N).
\end{eqnarray}
Then we get
\begin{eqnarray}
  \inf_{\partial\Omega}u_{\nu\nu}\geq -C-\sigma N.
\end{eqnarray}
\end{proof}

Then we prove Theorem \ref{th5.1} immediately.
\begin{proof}[\bf Proof of Theorem \ref{th5.1}]
 We choose $\sigma=\frac{1}{2}$ in Lemma \ref{le4.3} and \ref{le4.4}, then
\begin{eqnarray}
  \sup_{\partial\Omega}|u_{\nu\nu}|\leq C.\label{5.10}
\end{eqnarray}
Combining (\ref{5.10}) with (\ref{4.1}) in Lemma \ref{le4.1}, we obtain
\begin{eqnarray}
  \sup_{\overline{\Omega}}|D^2u|\leq C.
\end{eqnarray}
\end{proof}

\section{Existence of the Neumann boundary problem}\label{sec5}
We use the  method of continuity to prove the existence theorem for the Neumann problem (\ref{eq}).
\begin{proof}[\textbf{Proof of Theorem \ref{th1.1}}]
Consider a family of equations with parameter $t$,
\begin{equation}
\left\{
\begin{aligned}\label{eq6}
 &S_{k}(W)=tf+(1-t)\frac{(C_n^m)!m^k}{(C_n^m-k)!k!},\quad \text{in}\ \Omega,\\
 &u_{\nu}=-u+t\phi+(1-t)(x\cdot\nu+\frac{1}{2}x^2),\quad \text{on}\ \partial\Omega.
\end{aligned}
\right.
\end{equation}
From Theorem \ref{th3.1} and \ref{th5.1}, we get a glabal $C^2$ estimate independent of $t$
for the equation (\ref{eq6}).
 It follows that the equation (\ref{eq6}) is uniformly elliptic.
 Due to the concavity of $S_k^{\frac{1}{k}}(W)$ with respect to $D^2u$ (see \cite{cns2}),
 we can get the global H\"older estimates of second derivatives following the discussions
in \cite{lt}, that is, we can get
\begin{eqnarray}
  |u|_{C^{2,\alpha}}\leq C,
\end{eqnarray}
where $C$ depends only on $n$, $m$, $k$, $|u|_{C^{1}}$,$|f|_{C^{2}}$,$\min f$, $|\phi|_{C^{3}}$ and $\Omega$.
 It is easy to see that $\frac{1}{2}x^2$ is a $k$-admissible solution to (\ref{eq6}) for $t=0$.
 Applying the method of continuity (see \cite{gt}, Theorem 17.28), the existence of the classical
solution holds for $t=1$. By the standard regularity theory of uniformly elliptic partial differential
equations, we can obtain the higher regularity.
\end{proof}

\end{document}